\documentclass[12pt]{amsart}
\usepackage{amssymb}
\usepackage{hyperref}
\usepackage[all]{xy}
\usepackage{graphicx}

\theoremstyle{plain}
\newtheorem{theorem}{Theorem}[section]
\newtheorem{proposition}[theorem]{Proposition}

\newtheorem{lemma}[theorem]{Lemma}

\newtheorem{definition}[theorem]{Definition}

\theoremstyle{remark}
\newtheorem{remark}[theorem]{Remark}
\newtheorem{example}[theorem]{Example}

\numberwithin{equation}{section}

\newcommand{\K}{\ensuremath{\Bbbk}}

\newcommand{\C}{\mathcal C}

\newcommand{\oa}{\overline \alpha}

\newcommand{\og}{\overline \gamma}

\newcommand{\oR}{\overline R}
\newcommand{\oL}{\overline L}
\newcommand{\ow}{\overline w}

\renewcommand{\P}{\mathbb P}

\newcommand{\rad}{\operatorname{rad}}
\newcommand{\Hom}{\operatorname{Hom}}
\newcommand{\Ker}{\operatorname{Ker}}

\newcommand{\Ext}{\operatorname{Ext}}
\renewcommand{\Im}{\operatorname{Im}}

\newcommand{\Id}{\textsl{Id}}

\renewcommand{\mod}{\operatorname{mod}}
\newcommand{\HH}{\ensuremath{\mathsf{HH}}}

\newcommand{\achico}[1]{\scalebox{0.8}{${#1}$} }

\begin{document}

\title[Morita invariance for infinitesimal deformations]{Morita invariance for infinitesimal deformations}

\author[M. J. Redondo]{Mar\'\i a Julia Redondo}

\author[L. Rom\'an]{Lucrecia Rom\'an}

\author[F. Rossi Bertone]{Fiorela Rossi Bertone}

\author[M. Verdecchia]{Melina Verdecchia}

\address{Instituto de Matem\'atica (INMABB), Departamento de Matem\'atica, Universidad Nacional del Sur (UNS)-CONICET, Bah\'\i a Blanca, Argentina}
\email{mredondo@uns.edu.ar}
\email{lroman@uns.edu.ar}
\email{fiorela.rossi@uns.edu.ar}
\email{mverdec@uns.edu.ar}



\thanks{The first author is a research member of CONICET (Argentina). The first and second author  have been  supported  by  the  project  PICT-2015-0366.}

\date{\today}

\begin{abstract}
Let $A$ and $B$ be two Morita equivalent finite dimensional associative algebras over a field  $\K$. It is well known that Hochschild cohomology is invariant under Morita equivalence.  Since infinitesimal deformations are connected with the second Hochschild cohomology group, we explicitly describe the transfer map connecting $\HH^2(A)$ with $\HH^2(B)$.  This allows us to transfer Morita equivalence between $A$ and $B$ to that between infinitesimal deformations of them.
As an application, when $\K$ is algebraically closed, we consider the quotient path algebra associated to $A$ and describe the presentation by quiver and relations of the infinitesimal deformations of $A$. 
\end{abstract}

\maketitle

\section*{Introduction}\label{sect:introduction}

The algebraic deformation theory of associative algebras was introduced by Gerstenhaber in 1960s. In a series of papers he studied and described, between many other properties, the connection between deformations of an associative algebra and its Hochschild cohomology.

It is well known that Hochschild cohomology is invariant under Morita equivalence. It is natural to ask then, the behaviour of deformations related to Morita equivalences.

In this work we restrict our attention to infinitesimal deformations.
Set $\K$ a field and consider the ring of dual numbers $\K[t]/(t^2)$.
An infinitesimal deformation of an associative $\K$-algebra $A$ is an associative structure of $\K[t]/(t^2)$-algebra on $A[t]/(t^2)$ such that, modulo the ideal generated by $t$, the multiplication corresponds to that on $A$.
Gerstenhaber showed in \cite{G1} that the infinitesimal deformations of $A$ are parametrized by the second Hochschild cohomology group $\HH^2(A)$ of $A$ with coefficients in itself.

Let $A$ and $B$ be two Morita equivalent $\K$-algebras. The goal of the present paper is to describe how one can transfer the Morita equivalence between the algebras to that between the correspondent infinitesimal deformations of them.
For this, we give an explicit transfer map which assigns to each Hochschild $2$-cocycle of $A$, a Hochschild $2$-cocycle of $B$.
More generally, we define the transfer map as a cochain complex map and show that it induces isomorphisms $\HH^n(A)\simeq \HH^n(B)$ in each degree $n\geq 0$.   The definition and properties of this transfer map hold also in the more general context of stable equivalence of Morita type.  Similar transfer maps have been considered in \cite{B,Keller,loday, Z} for homology, and in \cite{Li, KLZ} for cohomology of symmetric algebras. 

The description of the finite dimensional modules over an infinitesimal deformation of $A$ that we develop in order to prove the Morita equivalence between deformations, would be the starting point for the study of the module category as well as its relation with the category of modules over the original algebra $A$. \\

The paper is divided into five sections. In the first one we introduce the needed concepts and notation. The second section is devoted  to describe explicitly the transfer map and show that it induces isomorphisms in cohomology.  Most important for our purpose is the second Hochschild cohomology group, which is intimately related with deformations. The goal of Section 3 is
to get a nice description of the category of modules over an infinitesimal deformation of an algebra.  In Section 4 we prove our main theorem concerning Morita invariance of deformations of algebras by constructing the bimodules that realize this equivalence. Last section applies the previous results in order to get a description by quiver and relations of any infinitesimal deformation of a finite dimensional algebra over an algebraically closed field $\K$.

\section{Preliminaries}\label{sect:preliminaries}

\subsection{Morita equivalent algebras}\label{subsec:Morita}

Let $A$ and $B$ be Morita equivalent algebras. That is, there exist bimodules $_AP_B, _BQ_A$,  and isomorphisms of bimodules
\begin{align}\label{eq: isos A y B}
	\langle - ,- \rangle_A : P \otimes_B Q \to A \quad \mbox{ and } \quad \langle -,-\rangle_B: Q \otimes_A P \to B.
\end{align}
Moreover, from \cite[eq.~(1.2.7.1), page 18]{loday}, we can assume that these isomorphisms satisfy, for all $p,p'\in P$ and $q,q'\in Q$, the following equations
\begin{align}
	\langle p, q\rangle_A p' & = p \langle q, p'\rangle_B,  \label{eq:cond_isos1}\\
	\langle q, p\rangle_B q' & = q \langle  p, q'\rangle_A.
	\label{eq:cond_isos2}
\end{align}
From now on, we fix the notation
\begin{align}\label{eq: 1 A y 1 B}
	1_A= \sum_{i=1}^{m'} \langle p'_i, q'_i\rangle_A \quad \mbox{ and } \quad 1_B= \sum_{k=1}^{m} \langle q_k, p_k\rangle_B.
\end{align}

\subsection{Hochschild cohomology}\label{subsect:HochCohom}

Given an algebra $A$ and $M$ an $A$-bimodule,
the Hochschild complex is the complex 
\begin{align*}
	0\to M\stackrel{d^1}{\longrightarrow} \Hom_\K (A, M) \stackrel{d^2}{\longrightarrow} \cdots  \longrightarrow \Hom_\K (A^{\otimes n}, M)\stackrel{d^{n+1}}{ \longrightarrow}  \Hom_\K (A^{\otimes (n+1)}, M) \longrightarrow \cdots
\end{align*}
where, for each $n>0$, $A^{\otimes n}$ denotes the $n$-fold tensor product of $A$ with itself over $\K$. The map $d^1:M\to\Hom_\K (A, M)$ is defined by $d^1(m)(a)= am-ma$, for $m\in M$, $a\in A$, and $d^{n+1}= \sum_{j=0}^{n+1}(-1)^{j} d^{n+1}_j$ is defined by
\begin{align*}\notag
	d_0^{n+1}(f)(a_0\otimes\cdots\otimes a_n) & = a_0f(a_1\otimes\cdots\otimes a_n), \\ 
	d_j^{n+1}(f)(a_0\otimes\cdots\otimes a_n)  & = f(a_0\otimes\cdots\otimes a_{j-1}a_j\otimes\cdots\otimes a_n), \quad 1 \leq j \leq n,\\ 
	d_{n+1}^{n+1}(f)(a_0\otimes\cdots\otimes a_n) & =f(a_0\otimes\cdots\otimes a_{n-1})a_n,
\end{align*}
for a $\K$-linear map $f : A^{\otimes n}\to M$ and $a_0,\ldots,a_n\in A$.
The $n$-th cohomology group of this complex is called the \emph{$n$-th Hochschild cohomology group of $A$} with coefficients in $M$, and it is denoted by $\HH^n(A,M)$.
We recall that, since $\K$ is a field, the Hochschild cohomology groups $\HH^n(A,M)$ can be identified with the groups $\Ext^n_{A-A}(A, M)$.  For any $f \in \Ker d^{n+1}$, we denote $[f] \in \HH^n(A,M)$.

If  $_AM_A={}_AA_A$,
then we write $\HH^n(A)$. It is well known that the Hochschild cohomology groups of two Morita equivalent algebras $A$ and $B$ are isomorphic, see for instance \cite[\S 1.5]{loday}.

When $A=\K Q/I$ is given by quiver and relations, see Subsection \ref{subsect: quiver}, a substantial reduction in the size of the complex used to compute $\HH^n(A,M)$ can be obtained by replacing $\Hom_\K (A^{\otimes n}, M)$ by 
$\Hom_{E-E} (\rad A^{\otimes_E n}, M)$, where $E = \K Q_0$ is the semisimple subalgebra of $A$ generated by the set of vertices $Q_0$ such that 
$A=E\oplus\rad A$, see \cite[Proposition 2.2]{Cibils}. 

For further definitions and facts, we refer the reader to \cite{W}.

\subsection{Deformations of associative algebras }\label{subsect:DeforAlg}

Let $A$ be a $\K$-algebra. We consider the truncated polynomial ring $\K[t]/(t^2)$. An \emph{infinitesimal deformation} of an associative $\K$-algebra $A$ is an associative structure of $\K[t]/(t^2)$-algebra on $A[t]/(t^2)$ such that modulo the ideal generated by $t$, the multiplication corresponds to that on $A$. More precisely:
\begin{definition}\label{def:inf def}
	Let  $f \in \Hom_\K(A \otimes_\K A, A)$. Let $A_f:= A[t]/(t^2)\simeq A\oplus A$ be the algebra with multiplication  given by
	\begin{align*}
		(a_0, b_0)(a_1, b_1) = (a_0 a_1, a_0 b_1 + b_0 a_1+ f(a_0 \otimes a_1)), \qquad \forall (a_0, b_0), (a_1, b_1) \in A[t]/(t^2).
	\end{align*}
	If this multiplication is associative we say that $A_f$ is an \emph{infinitesimal deformation of $A$}.
\end{definition}
We say that two infinitesimal deformations $A_f$ and $A_{f'}$ are \emph{equivalent} if there exists a $\K[t]/(t^2)$- linear map $\phi: A_f\to A_{f'}$ such that  $\phi (a,0)=(a,0)$ and 
\begin{align*}
	\phi\circ \operatorname{mult}_{A_f} =\operatorname{mult}_{A_{f'}}\circ (\phi\otimes \phi).
\end{align*} 
If we write out the associativity condition in $A_f$ we have
$$a_0f(a_1\otimes a_2)+f(a_0\otimes a_1a_2) = f(a_0a_1\otimes a_2) + f(a_0\otimes a_1)a_2$$
for all $a_0,a_1,a_2\in A$, that is, $f $ is a Hochschild $2$-cocycle. Moreover, the next correspondence is well known and was proved in \cite{G1}.
\begin{theorem}
	There is a one-to-one correspondence between the space of equivalence classes of infinitesimal deformations of $A$ and the second Hochschild cohomology group $\HH^2(A)$.
\end{theorem}

\subsection{Presentation by quiver and relations}\label{subsect: quiver}

We recall some basic concepts about quivers and algebras; for unexplained notions and further results we refer, for instance, to \cite{ARS, ASS}.

A finite \emph{quiver} $Q=(Q_0,Q_1)$ is a finite oriented graph with set of vertices $Q_0$ and set of arrows $Q_1$. We denote by $s,t:Q_1\rightarrow Q_0$ the \emph{source} and \emph{target} maps respectively. 
A \emph{path} $w$ in $Q$ of length $n$ is a concatenation of arrows $w=\alpha_1\cdots\alpha_n$ with $t(\alpha_i) = s(\alpha_{i+1})$ for $1\leq i < n$; for any vertex $i\in Q_0$,  $e_i$ is the trivial path of length zero.  We put $s(w)=s(\alpha_1)$ and $t(w)=t(\alpha_n)$, and $s(e_i)=i=t(e_i)$.
We say that an arrow $\alpha$ divides a path $w$ if $w= L(w) \alpha R(w)$, where $L(w)$ and $R(w)$ 
are paths in $Q$.

The \emph{path algebra} $\K Q$ is the $\K$-algebra whose basis is the set of all paths in $Q$, including one trivial path $e_i$ at each vertex $i\in Q_0$, endowed with the multiplication induced from the concatenation of paths.
The sum of the trivial paths is the identity.

It is well known that in case $\K$ is algebraically closed, any finite-dimensional $\K$-algebra $A$ is Morita equivalent to a quotient of a path algebra $\K Q/I$, where $Q$ is a quiver and $I$ an admissible ideal of $\K Q$. In this case,
the pair $(Q, I)$ is called a \emph{presentation of $A$ by quivers and relations}.
A \emph{relation} $\rho$ on the ideal $I$ is a $\K$-linear combination of paths $\rho=  \sum_{i} \ \lambda_i \gamma_i$ with $\lambda_i$ nonzero scalars and $\gamma_i$ paths of length at least two having all the same source and the same target.

\section{Transfer map}\label{sec: transfer map} 
From now on, let $A$ and $B$ be Morita equivalent algebras. It is well known that the Hochschild cohomology groups of $A$ and $B$ are isomorphic, see for instance \cite[\S 1.5]{loday}. The goal of this section is to give an explicit transfer map $\HH^n(A)\to\HH^n(B)$. This map when $n=2$ will be crucial in Section \ref{sec:morita_eq}. \\

Recall the notation from \eqref{eq: 1 A y 1 B}, $1_B= \sum_{i=1}^m \langle q_i, p_i\rangle_B$. 
Define $\phi^n:\Hom_\K (A^{ \otimes n}, A)\to \Hom_\K (B^{ \otimes n}, B)$ by
\begin{align}\label{eq:def_t_n}
	\phi^n(f)(b_1 \otimes \cdots \otimes b_n)= \sum_{1 \leq i_0, \cdots, i_n \leq m} \langle q_{i_0}, f(\langle p_{i_0},b_1q_{i_1}\rangle_A \otimes \cdots \otimes \langle p_{i_{n-1}},b_nq_{i_n}\rangle_A)p_{i_n}\rangle_B,
\end{align}
for  $f\in \Hom_\K (A^{\otimes n},  A)$, $b_1,\cdots, b_n\in B$.

\begin{proposition}\label{prop:complex}
	Let $A$ and $B$ be Morita equivalent algebras. Then the maps $\phi^n$
	induce maps $$\widetilde {\phi^n}:\HH^n(A)\to \HH^n(B).$$
\end{proposition}
\begin{proof}
	It is evident that $(\phi^n): (\Hom_\K (A^{\otimes n},A), d^{n+1}) \to (\Hom_\K (B^{\otimes n},B), d^{n+1}) $ is a cochain map since the equations
	\begin{align*}
		&\sum_{i_0=1}^m q_{i_0}  \langle p_{i_0}, b_0 q_{i_1} \rangle_A  = b_0 q_{i_1}, \\
		&\sum_{i_{j-1}=1}^m  \langle p_{i_{j-2}}, b_{j-1} q_{i_{j-1}} \rangle_A \langle p_{i_{j-1}}, b_{j} q_{i_j} \rangle_A = \langle p_{i_{j-2}}, b_{j-1} b_{j} q_{i_j} \rangle_A, \\
		&\sum_{i_{n+1}=1}^m   \langle p_{i_n}, b_n q_{i_{n+1}} \rangle_A p_{i_{n+1}} = p_{i_n} b_n, 
	\end{align*} 
	imply that $d_j^{n+1} \phi^n = \phi^{n+1} d_j^{n+1}$ for all $j$ such that $0 \leq j \leq n+1$. Hence each $\phi^n$ induces a well-defined map $\widetilde \phi^n:\HH^n(A)\to \HH^n(B)$.
	\end{proof}

It is clear that $\phi^n$ is defined in terms of $1_B= \sum_{i=1}^m \langle q_i, p_i\rangle_B$. Assume that $1_B= \sum_{i=1}^{\hat m} \langle \hat q_i, \hat p_i\rangle_B$ and let $\hat \phi^n$ be the transfer map defined in terms of $1_B= \sum_{i=1}^{\hat m} \langle \hat q_i, \hat p_i\rangle_B$.

\begin{proposition}\label{prop:homotopia}
	Let $A$ and $B$ be Morita equivalent algebras. Then the maps $\phi^n$ and $\hat \phi^n$ are homotopic.
\end{proposition}

\begin{proof}
Let  $h^{n+1}: \Hom_\K (A^{\otimes n+1},A) \to \Hom_\K (B^{\otimes n}, B)$ be given by 
	\begin{align}\label{eq:def_h}
		h^{n+1}= \sum_{r=1}^{n+1} (-1)^r h^{n+1}_r
	\end{align}
	with, for $2 \leq r \leq n$,
		 	\begin{align*}	
		h^{n+1}_1 (f)  (b_1 \otimes \cdots \otimes b_n ) =   \sum_{1 \leq i_0, \cdots, i_{n+1} \leq m} &
		\langle q_{i_0} , f(\langle  p_{i_0}, \hat q_{i_1}\rangle_A \otimes \langle \hat p_{i_1},b_1q_{i_2}\rangle_A \\
		& \otimes \langle p_{i_2},b_2q_{i_3}\rangle_A \otimes \cdots 
		 \otimes \langle p_{i_{n}},b_nq_{i_{n+1}}\rangle_A )  p_{i_{n+1}} \rangle_B,\\
			h^{n+1}_r(f) (b_1 \otimes \cdots \otimes b_n) 
		= \sum_{1 \leq i_0, \cdots, i_{n+1} \leq m}   &\langle \hat q_{i_0} , f(\langle \hat p_{i_0},b_1\hat q_{i_1}\rangle_A \otimes \cdots  
		 \otimes \langle \hat p_{i_{r-3}},b_{r-2} \hat q_{i_{r-2}}\rangle_A \\	 
		 & \otimes  \langle \hat  p_{i_{r-2}},b_{r-1} q_{i_{r-1}} \rangle_A  
		\otimes  \langle   p_{i_{r-1}},  \hat q_{i_{r}}\rangle_A \otimes  \langle \hat  p_{i_{r}},b_{r} q_{i_{r+1}} \rangle_A \\	 
		&
		 \otimes  \langle  p_{i_{r+1}},b_{r+1} q_{i_{r+2}} \rangle_A  \otimes \cdots \otimes \langle p_{i_{n}},b_nq_{i_{n+1}}\rangle_A )  p_{i_{n+1}}  \rangle_B, \\ \notag
		h^{n+1}_{n+1} (f)  (b_1 \otimes \cdots \otimes b_n )  = \sum_{1 \leq i_0, \cdots, i_{n+1} \leq m} &\langle \hat q_{i_0},  f(\langle \hat p_{i_0},b_1\hat q_{i_1}\rangle_A \otimes \cdots \otimes \langle \hat p_{i_{n-1}},b_nq_{i_n}\rangle_A \\
		& \otimes \langle   p_{i_{n}}, \hat q_{i_{n+1}}\rangle_A ) \hat p_{i_{n+1}} \rangle_B.
		\end{align*}

		Tedious but direct computations show that 
\begin{align*}
  h_1^{n+1}d_1^{n+1} &= \phi^n;  &                      && && h_r^{n+1}d_{n+1}^{n+1} &= d^n_n h_r^n,     && \achico{1 \leq r\leq n}; \\
  h_{n+1}^{n+1}d_n^{n+1} &=  \hat \phi^n;  &       && && h_r^{n+1}d_i^{n+1}     &= d^n_i h^n_{r-1}, && \achico{3 \leq r\leq n+1, \ 1\leq i\leq r-2}; \\
  h_r^{n+1} d_0^{n+1} &= d^n_0 h^n_{r-1}, && \achico{2 \leq r\leq n+1}; &&& h_r^{n+1}d_j^{n+1} &= d^n_{j-1}h_r^n, && \achico{1 \leq r\leq n-1, \ r+1\leq j\leq n};\\
  h_r^{n+1}d_{r-1}^{n+1} &= h_{r+1}^{n+1}d_{r+1}^{n+1}, && \achico{1 \leq r\leq n}. 
\end{align*}
Then one can check that
\begin{align*}
		h^{n+1} d^{n+1} & = \sum_{r=1}^{n+1} \sum_{i=0}^{n+1} (-1)^{r+i} h_r^{n+1} d_i^{n+1} 
		= h_1^{n+1} d_1^{n+1} - h_{n+1}^{n+1} d_n^{n+1} \\
		& \quad +  \sum_{r=2}^{n+1}  (-1)^r h_r^{n+1} d_0^{n+1} 
		+  \sum_{r=1}^{n} (-1)^{r+n+1} h_r^{n+1} d_{n+1}^{n+1}   
		-  \sum_{r=1}^{n} h_r^{n+1} d_{r-1}^{n+1}  \\
		& \quad +  \sum_{r=2}^{n+1}  h_r^{n+1} d_r^{n+1} +  \sum_{r=3}^{n+1} \sum_{i=1}^{r-2} (-1)^{r+i} h_r^{n+1} d_i^{n+1} 
		+  \sum_{r=1}^{n-1} \sum_{i=r+1}^{n} (-1)^{r+i} h_r^{n+1} d_i^{n+1} \\
		& =  \phi^n - \hat \phi^n  +  \sum_{r=2}^{n+1}  (-1)^r  d_0^{n} h_{r-1}^{n}
		+  \sum_{r=1}^{n} (-1)^{r+n+1}  d_{n}^{n}  h_r^{n} \\
		& \quad +   \sum_{r=3}^{n+1} \sum_{i=1}^{r-2} (-1)^{r+i} d_i^{n} h_{r-1}^{n} 
		+  \sum_{r=1}^{n-1} \sum_{i=r+1}^{n} (-1)^{r+i} d_{i-1}^{n} h_r^{n}  
\end{align*}
and this equals to
\begin{multline*}
		\phi^n -  \hat \phi^n  -  \sum_{r=1}^{n}  (-1)^r  d_0^{n} h_{r}^{n}
		-  \sum_{r=1}^{n} (-1)^{r+n}  d_{n}^{n}  h_r^{n} \\
		-   \sum_{r=2}^{n} \sum_{i=1}^{r-1} (-1)^{r+i} d_i^{n} h_{r}^{n} 
		-  \sum_{r=1}^{n-1} \sum_{i=r}^{n-1} (-1)^{r+i} d_{i}^{n} h_r^{n} 
		= \phi^n - \hat \phi^n  - d^n h^n.
\end{multline*}
	Hence, we have
	\begin{align}\label{eq: h homotopy}
		h^{n+1} d^{n+1}= \phi^n - \hat \phi^n - d^n h^n
	\end{align}
	and the proof of the proposition is  complete.
\end{proof}

Analogously,  for $1_A= \sum_{j=1}^{m'} \langle p'_j, q'_j\rangle_A$, we have $\psi^n:\Hom_\K (B^{\otimes n}, B)\to \Hom_\K (A^{\otimes n}, A)$ given by
\begin{align}\label{eq:def_s_n}
	\psi^n(g)(a_1 \otimes \cdots \otimes a_n)= \sum_{1 \leq j_0, \cdots ,j_n \leq m'} \langle p'_{j_0}, g(\langle q'_{j_0},a_1p'_{j_1}\rangle_B \otimes \cdots \otimes \langle q'_{j_{n-1}},a_np'_{j_n}\rangle_B)q'_{j_n}\rangle_A,
\end{align}
for $g\in \Hom_\K (B^{\otimes n}, B)$ and $a_1,\cdots, a_n \in A$.

\begin{theorem}\label{thm:transfer}
	Let $A$ and $B$ be Morita equivalent algebras. Then the induced maps $$\widetilde {\phi^n}:\HH^n(A)\to \HH^n(B) \, \mbox{ and } \, \widetilde {\psi^n}:\HH^n(B)\to \HH^n(A)$$ are isomorphisms and inverse to each other.
\end{theorem}
\begin{proof}
A direct computation shows that the composition $\psi^n \circ \phi^n$ is the transfer map defined in terms of 
$$1_A = \sum_{i=1}^m \sum_{j=1}^{m'} \langle \langle p'_j \otimes q_i , p_i \otimes q'_j \rangle \rangle_A$$
where $\langle \langle - , - \rangle \rangle_A : P \otimes _B Q \otimes_A P \otimes_B Q \to A$ is given by
$$\langle \langle p \otimes q ,  p' \otimes q' \rangle \rangle_A = \langle p ,  \langle q ,  p' \rangle_B q' \rangle_A =
\langle p ,  q \rangle_A \langle p' , q' \rangle_A.$$
The last equality shows that $$1_A = \sum_{i,j=1}^{m'} \langle \langle p'_j \otimes q'_j , p'_i \otimes q'_i \rangle \rangle_A$$
and the transfer map defined in terms of this equality is the identity since 
$$\sum_{i,j}  \langle \langle p'_i \otimes q'_i , a ( p'_j \otimes q'_j )  \rangle \rangle_A =a.$$
Hence the result follows from Proposition \ref{prop:homotopia}.
\end{proof}

\begin{remark} The transfer defined above can be compared with the one obtained in \cite{KLZ} for symmetric algebras since for any $x \in P$ we have that
	\begin{align*}
		x= 1_A x = \sum_{i=1}^{m'} \langle p'_i, q'_i\rangle_A x = \sum_{i=1}^{m'} p'_i \langle q'_i, x \rangle_B = \sum_{i=1}^{m'} p'_i \varphi_i(x)
	\end{align*}
	where $\varphi_i = \langle q'_i, - \rangle_B$ belongs to $\Hom_B(P,B)$.
\end{remark}

\begin{remark}\label{stable} Concerning invariance of Hochschild cohomology with respect to equivalences, it is well known that it is not invariant under stable equivalences.  However, the transfer defined above can be defined exactly in the same way in the stable equivalence of Morita type context by considering epimorphisms of bimodules
\begin{align*}
	\langle - ,- \rangle_A : P \otimes_B Q  \stackrel{\sim} {\longrightarrow} A \oplus X \to A
  \quad \mbox{ and } \quad \langle -,-\rangle_B: Q \otimes_A P  \stackrel{\sim} {\longrightarrow}  B \oplus Y \to B
\end{align*}
where $X$ is a projective $A$-bimodule and $Y$ is a projective $B$-bimodule, and assuming that these epimorphisms satisfy, for all $p,p'\in P$ and $q,q'\in Q$, the equations
$$
	\langle p, q\rangle_A p'  = p \langle q, p'\rangle_B,  \qquad \mbox{ and } \qquad 
	\langle q, p\rangle_B q'  = q \langle  p, q'\rangle_A.$$
	 Moreover, for any $A$-bimodule $M$, the transfer map $\widetilde \phi^n: \HH^n (A,M) \to \HH^n(B, Q \otimes_A M \otimes_A P)$ is induced by 
	\begin{align*}
	\phi^n(f)(b_1 \otimes \cdots \otimes b_n)= \sum_{1 \leq i_0, \cdots, i_n \leq m} q_{i_0} \otimes  f(\langle p_{i_0},b_1q_{i_1}\rangle_A \otimes \cdots \otimes \langle p_{i_{n-1}},b_nq_{i_n}\rangle_A) \otimes p_{i_n},
\end{align*}
for  $f\in \Hom_\K (A^{\otimes n}, M)$, $b_1,\cdots, b_n\in B$.
	
	The version of these transfer maps for Hochschild homology can be found in \cite{B} and in \cite[Section 5.8]{Z}.

\end{remark}

\section{$A_f$-Modules}\label{sec:Af modules}

Let $A$  and $B$ be two Morita equivalent $\K$-algebras, let $f : A \otimes_\K A \to A$ be a Hochschild $2$-cocycle  and $g=\phi^2(f)$. 
The goal of this paper is to show the Morita equivalence between the infinitesimal deformations $A_f$ and $B_{g}$. In order to prove this equivalence, in Section \ref{sec:morita_eq} we will construct $A_f$-$B_g$-bimodules. An $A_f$-$B_g$-bimodule is an uple $(M_0,M_1,T_M,f_M,g_M)$ where $M_0$ and $M_1$ are $A$-$B$-bimodules, $(M_0,M_1,T_M,f_M)$ is a left $A_f$-module, $(M_0,M_1,T_M,g_M)$ is a right $B_g$-module and these module structures satisfy certain compatible condition, see Remark \ref{bimodulos}.
We need first to describe the category of $A_f$-modules.
In this section we introduce a category $\C_f$ and we describe the category $\mod A_f$ of finite dimensional left modules over the infinitesimal deformation $A_f$ by means of an equivalence functor $F:\C_f \to \mod A_f$.

Let $\C_f$ be the category whose objects are  uples $(M_0,M_1,T_M,f_M)$ with
$M_0, M_1 \in \mod A$, $T_M \in \Hom_A(M_0, M_1)$ a monomorphism and $f_M \in \Hom_\K (A \otimes_\K M_0, M_1)$, satisfying the following condition
\begin{align} \label{eqn:1a}
	af_M(b \otimes m_0) - f_M (ab \otimes m_0) + f_M(a \otimes bm_0) - f(a \otimes b) T_M(m_0) =0
\end{align}
for  $a,b \in A, m_0 \in M_0$. The morphisms between objects $(M_0,M_1,T_M,f_M)$ and $(N_0,N_1,T_N,f_N)$ are triples of morphisms $(u_0, u_1, u_2)$ where $u_0 \in \Hom_A(M_0, N_0)$, $u_1 \in \Hom_\K (M_0, N_1)$ and $u_2 \in \Hom_A (M_1, N_1)$ are such that the diagram
\begin{align*}
	\xymatrix{
		M_0 \ar[d]_{T_M} \ar[r]^{u_0} & N_0  \ar[d]^{T_N}\\
		M_1  \ar[r]^{u_2}  & N_1 }
\end{align*}
commutes and
\begin{align}\label{eq:u_1}
	u_1(a m_0) & = a u_1 (m_0) - u_2 (f_M(a \otimes m_0)) + f_N(a \otimes u_0(m_0)), \qquad \forall a\in A, \, m_0\in M_0.
\end{align}
The composition is given by
$$(v_0, v_1, v_2) (u_0, u_1, u_2)= (v_0 u_0, v_2 u_1 + v_1 u_0, v_2 u_2)$$
and the identity morphism is $(\Id, 0, \Id)$. It is easy to check that $\C_f$ is a category. \\

Define a functor $F: \C_f \to \mod A_f$ as follows: 
\begin{itemize}
	\item For $(M_0,M_1,T_M,f_M)$ in $\C_f$, let
	$F(M_0,M_1,T_M,f_M)= M_0 \oplus M_1$ as $\K$-vector space. The $A_f$-module structure is given by 
	\begin{align*}
		(a,b)(m_0, m_1) = (am_0, am_1 + b T_M(m_0) + f_M(a \otimes m_0))
	\end{align*}
	for $a,b \in A, m_0 \in M_0$ and $m_1 \in M_1$. 
	\item For $(u_0, u_1, u_2)$ a morphism in $\C_f$, let $F(u_0, u_1, u_2)=u$ with 
	\begin{align*}
		u(m_0,m_1) = (u_0(m_0) , u_1(m_0) + u_2(m_1))
	\end{align*}
	for $m_0 \in M_0$ and $m_1 \in M_1$. 
\end{itemize}

Summing up the above considerations, we get the following result that gives the formal connection between $\C_f$ and $\mod A_f$ via the functor $F$.

\begin{proposition}
	The functor $F: \C_f \to \mod A_f$ defined above is an equivalence of categories.
\end{proposition}
\begin{proof} 
	It is easy to check that $F$ is a functor and that it is faithful. 
	To see that it is dense, let $M$ be an $A_f$-module. Define $T: M \to M$ by $T(m)=(0,1) m$.
	Set $M_1= \Ker T$. Let $M_0$ be a complement of the vector subspace $M_1$ of $M$, that is, $M=M_0 \oplus M_1$ as $\K$-vector space. 	
	Then $M_1$ is an $A$-module with the action induced by the $A_f$-module structure of $M$, that is,
	\begin{align*}
		a\cdot m=(a,0) m, \qquad \forall a\in A, \, m\in M_1.
	\end{align*}
	In fact, if $a,b \in A, m \in M_1$ then
	$$a \cdot (b \cdot m) = (a,0)(b,0)m = (ab, f(a\otimes b))m = (ab,0) m + (f(a \otimes b),0)(0,1)m= (ab)\cdot m,$$
	since $(0,1)m = T(m)=0$.
	Crearly $T(M_0) \subset M_1$ because $T^2 =0$. Set $T_M: M_0 \to M_1$ the injective map induced by $T$.
	
	Now we shall give an action of $A$ on $M_0$. Notice that $T|_{M_0}: M_0\to \Im T$ is an isomorphism.
	Hence, we have $T':=(T|_{M_0})^{-1}: \Im T \to M_0$, in particular $T(T'(m))=m$ for all $m\in \Im T$.
	Define
	\begin{align}\label{eq:action A in M0}
		a*m=T'(a\cdot T(m)), \qquad \forall a\in A, \,m\in M_0.
	\end{align}
	This gives an action of $A$ on $M_0$. Indeed,
	\begin{align*}
		a*(b*m) &=a*(T'(b\cdot T(m)))=T'(a \cdot T(T'(b\cdot T(m))))\\
		&= T'(a\cdot (b\cdot T(m)))= T'(ab\cdot  T(m))=ab*m.
	\end{align*}
	Thus  $T_M:M_0\to M_1$ is a monomorphism of $A$-modules.
	
	The $\K$-linear map $f_M: A \otimes_\K M_0 \to M_1$ is defined as 
	\begin{align}\label{eq:f_M}
		f_M(a\otimes m):= (a,0)  m - a*m, 
	\end{align}
	for $a\in A$ and $m\in M_0$.
	Since 
	\begin{align*}
	T(f_M(a\otimes m))  = T((a,0) m) - T(a*m) 
	 = (0,1) (a,0) m - (a,0)(0,1) m =0,
	\end{align*} 
	we have that $f_M(a\otimes m)\in M_1$.
	Moreover, 
	\begin{align*}
		a  \cdot f_M(b\otimes m) &- f_M (ab\otimes m)+ f_M(a\otimes b*m)- f(a\otimes b) \cdot T_M(m) \\
		= &(a,0) ((b,0) m - b*m) - ((ab,0) m- ab*m) \\
		&+ ((a,0) (b*m)- a*(b*m))-f(a\otimes b)\cdot T(m)\\
		=& (ab,f(a\otimes b))m - (a,0)(b*m)- (ab,0) m +  ab*m \\
		& + (a,0)(b*m)-ab*m-f(a\otimes b)\cdot T(m)\\
		= &(f(a\otimes b),0)(0,1) m -f(a\otimes b)\cdot  T(m) =0.
	\end{align*}
	Finally, we show that the functor $F$ is full. For this, let 
	$$u: F(M_0,M_1,T_M,f_M) \to F(N_0,N_1,T_N,f_N)$$
	be a morphism in $\mod A_f$. For $m_0 \in M_0$ and $m_1 \in M_1$ we have
	$u(m_0, 0)=(n_0, n_1)$ and  $u(0,m_1)  = (n'_0, n'_1)$.
	In fact $n'_0=0$ since $0=  (0,1) u(0,m_1) = (0, T_N(n'_0))$ and $T_N$ is a monomorphism.
	Define
	\begin{align}
		u_0(m_0):=n_0, \qquad u_1(m_0):=n_1, \qquad u_2(m_1):=n_1'.
	\end{align}	
	Now, since
	\begin{align*}
		u(am_0, 0 ) & = u ((a,0)  (m_0, 0) )- u (0, f_M(a \otimes m_0) )\\
		& = (a,0) u (m_0,0) - u (0, f_M(a \otimes m_0) )\\
		& = (a n_0, an_1 + f_N (a \otimes n_0)) - u (0, f_M(a \otimes m_0)),
	\end{align*} 
	we have that $u_0(am_0)= an_0 = a u_0(m_0)$ and 
	\begin{align*}
		u_1(a m_0)& = a n_1 + f_N ( a \otimes n_0) - u_2 (f_M (a \otimes m_0)) \\
		& = a u_1 (m_0) + f_N ( a \otimes u_0(m_0)) - u_2 (f_M (a \otimes m_0)).
	\end{align*} 
	Hence $u_0$ is an $A$-morphism and $u_1$ satisfies \eqref{eq:u_1}. Moreover
	$$u(0, am_1)  = (a,0) (0, n'_1) = (0, an'_1)$$
	and then $u_2$ is an $A$-morphism. A direct computation shows that $u=F(u_0, u_1, u_2)$ and the proof is complete. 
\end{proof}

Likewise, one can describe right $A_f$-modules as uples $(M_0,M_1,T_M,f_M)$.

\begin{example}
	For the $A_f$-module $A_f$ we have that $F(A,A,\Id, f)=A_f$. In fact, 
	since $T: A_f \to A_f$ is given by $T(a,b)=(0,1)(a,b)=(0,a)$  for all $a,b\in A$, then $M_1=\{(0,b):b\in A\}$ and the complement $M_0=\{(a,0):a\in A\}$. The morphism $T'(0,b)=(b,0)$ for $b\in A$, hence $M_0$ is an $A$-module via $a*(a',0)=T'((a,0)T(a',0))=T'(0,aa')=(aa',0)$,  for all $a,a'\in A$.
	Finally,  $f_{A_f}=f:A\otimes_\K A\to A$ since, for all $a, b \in A$,
	$$(a,0)(b,0)-a*(b,0)=(ab,f(a\otimes b))-(ab,0)=(0,f(a\otimes b)).$$
\end{example}

\begin{remark} \label{bimodulos}
	In subsequent sections we will need the notion of bimodules. Let $A$ and $B$ be algebras and let $f:A\otimes_\K A\to A$, $g:B\otimes_\K B\to B$ be Hochschild $2$-cocycles. Consider the infinitesimal deformations $A_f$ and $B_g$. 
	By abuse of notation, we omit the corresponding equivalent functor. So, an $A_f$-$B_g$-bimodule is an uple $(M_0,M_1,T_M,f_M,g_M)$ where $M_0$ and $M_1$ are $A$-$B$-bimodules, $(M_0,M_1,T_M,f_M)$ is a left $A_f$-module, $(M_0,M_1,T_M,g_M)$ is a right $B_g$-module and 
	the maps $f_M \in \Hom_\K (A \otimes_\K M_0, M_1)$ and $g_M \in \Hom_\K ( M_0\otimes_\K B, M_1)$ satisfy the compatible condition
	\begin{align} \label{eq:bimod_comp}
		g_M(am_0  \otimes b)+f_M(a  \otimes m_0)b=	ag_M(m_0  \otimes b)  + f_M(a  \otimes m_0 b)
	\end{align}
	for all $a\in A$, $b\in B$ and $ m_0\in M_0$.
\end{remark}

\section{Morita equivalence}\label{sec:morita_eq}
Let $A$ and $B$ be Morita equivalent algebras.  Fix $f$ a representative element of $\HH^2(A)$ and $g:= \phi^2(f)$ the associated representative element of $\HH^2(B)$ given by the transfer map \eqref{eq:def_t_n}. \\

Recall from Subsection \ref{subsec:Morita} that the Morita equivalence between $A$ and $B$ is given by two bimodules $_AP_B, _BQ_A$ and the isomorphisms of bimodules
$$\langle - ,- \rangle_A : P \otimes_B Q \to A \quad \mbox{ and }\quad  \langle -,-\rangle_B: Q \otimes_A P \to B$$
satisfying \eqref{eq:cond_isos1} and \eqref{eq:cond_isos2}.

This section is devoted to prove that the infinitesimal deformations $A_f$ and $B_g$ are Morita equivalent: we need to construct two bimodules $_{A_f}{\hat P}_{B_g}$, $_{B_g}{\hat Q}_{A_f}$  and the corresponding isomorphisms of bimodules
$$\langle - ,- \rangle_{A_f}: \hat P \otimes_{B_g} \hat Q \to A_f \quad \mbox{ and } \quad \langle -,-\rangle_{B_g}: \hat Q \otimes_{A_f} \hat P \to B_g.$$

\begin{remark} \label{key}
	The key point for describing the needed structure of bimodules is the following fact.
	By \eqref{eq: 1 A y 1 B}, we know that $\{p'_i\}_{1 \leq i \leq m'}$ is a set of generators of the $B$-module  $P$ and $\{p_k\}_{1 \leq k \leq m}$ is a set of generators of the $A$-module $P$.
	More precisely, for $x \in P$,
	\begin{align*}
		x &=  x 1_B = x \sum_{k=1}^{m} \langle q_k, p_k\rangle_B = \sum_{k=1}^{m}  \langle x, q_k \rangle_A p_k \\
		& = 1_A x = \sum_{i=1}^{m'} \langle p'_i, q'_i\rangle_A x = \sum_{i=1}^{m'}p'_i \langle q'_i, x \rangle_B.
	\end{align*}
\end{remark}

We define
$ \hat P=(_AP_B, _AP_B, \Id_P, f_P: A \otimes_\K P \to P, g_P: P \otimes_\K B \to P)$
with
\begin{align*}
	f_P(a \otimes p) & = \frac{1}{2}\sum_{i=1}^m f(a \otimes \langle p, q_i \rangle_A) p_i + \frac{1}{2}\sum_{1 \leq i_0, i_1 \leq m'} p'_{i_0} g( \langle q'_{i_0}, a p'_{i_1} \rangle_B \otimes \langle q'_{i_1} , p \rangle_B) + \frac{1}{2}  h^2 (f)(a)p,\\
	g_P(p \otimes b) & = \frac{1}{2}\sum_{1 \leq i_0, i_1 \leq m} f( \langle p, q_{i_0} \rangle_A \otimes \langle p_{i_0} b, q_{i_1} \rangle_A) p_{i_1} + \frac{1}{2}\sum_{i=1}^{m'} p'_i g(\langle q'_i, p \rangle_B \otimes b),
\end{align*}
where $h^2$ is the map defined by equation \eqref{eq:def_h} for the identity and $\psi\circ \phi$,  that is:
\begin{align*}	
	h^{2} (f)  (a ) = &   \sum_{\substack{1 \leq j_0,j_1 \leq m' \\  1 \leq i_0,i_1 \leq m }}  
	\langle p'_{j_0} , q_{i_0} \rangle_A f(\langle  p_{i_0},  q'_{j_0}\rangle_A a \otimes \langle  p'_{j_1},q_{i_1}\rangle_A )  \langle p_{i_1},q'_{j_1} \rangle_A\\
	& -  \sum_{\substack{1 \leq j \leq m' \\  1 \leq i \leq m }}  
	 f(\langle  p'_j,  q_i\rangle_A  \otimes \langle  p_i,q'_j \rangle_A a).
\end{align*}

Analogously, let
$ \hat Q=(_BQ_A, _BQ_A, \Id_Q, g_Q: B \otimes_\K Q \to Q, f_Q: Q \otimes_\K  A \to Q)$
with
\begin{align*}
	f_Q(q \otimes a) & = \frac{1}{2}\sum_{i=1}^m q_i f( \langle p_i, q \rangle_A \otimes  a)
	+ \frac{1}{2}\sum_{1 \leq i_0, i_1 \leq m'} g( \langle q,  p'_{i_0} \rangle_B \otimes  \langle q'_{i_0} , ap'_{i_1} \rangle_B) q'_{i_1}
	+ \frac{1}{2}q\, h^2(f)(a) ,\\
	g_Q(b \otimes q) & =  \frac{1}{2}  \sum_{1 \leq i_0, i_1 \leq m} q_{i_0}  f( \langle p_{i_0}  b q_{i_1} \rangle_A \otimes \langle p_{i_1} , q \rangle_A)
	+\frac{1}{2} \sum_{i=1}^{m'} g(b \otimes \langle q, p'_i \rangle_B)q'_i.
\end{align*}
\begin{lemma}\label{lem:hatP and hatQ} With the above notation,
	$\hat P$ is an $A_f$-$B_g$-bimodule and $\hat Q$ is an $B_g$-$A_f$-bimodule.
\end{lemma}
\begin{proof}
	We will only consider $\hat P$ since the same reasoning applies to $\hat Q$. It suffices to see that the following three equations are satisfied, see \eqref{eqn:1a} and \eqref{eq:bimod_comp},
	\begin{align}\label{eq:bimod1}
		a_0 f_P(a_1 \otimes  p) - f_P(a_0 a_1 \otimes p) + f_P(a_0 \otimes  a_1p) - f(a_0 \otimes  a_1) p =0 ;\\
		\label{eq:bimod2}
		p g(b_0 \otimes b_1) - g_P (p   b_0 \otimes  b_1) + g_P(p \otimes b_0  b_1) - g_P(p \otimes  b_0) b_1 =0;\\
		\label{eq:bimod3}
		ag_P(p \otimes  b) - g_P(ap \otimes b)+ f_P(a \otimes p b) - f_P(a \otimes  p)b =0.
	\end{align}
	We present an outline and some details of the needed computations. To see that $g_P$ satisfies \eqref{eq:bimod2}, we need to prove that both  $$\sum_{1 \leq i_0, i_1 \leq m} f( \langle p, q_{i_0} \rangle_A \otimes  \langle p_{i_0} b, q_{i_1} \rangle_A) p_{i_1} \quad \mbox{ and } \quad \sum_{i=1}^{m'} p'_i g(\langle q'_i, p \rangle_B \otimes b)$$ do. To verify the first one we add $1_B$ conveniently and apply \eqref{eq:cond_isos1}:
	\begin{align*}
		p \sum_{1 \leq i_0,i_1,i_2 \leq m} \langle q_{i_0}, f(\langle p_{i_0} , b_0q_{i_1}\rangle_A \otimes  \langle p_{i_1},b_1q_{i_2}\rangle_A)p_{i_2}\rangle_B - \sum_{1 \leq i_1,i_2 \leq m}
		f( \langle p 1_Bb_0, q_{i_1} \rangle_A \otimes \langle p_{i_1} b_1, q_{i_2} \rangle_A) p_{i_2}
		\\
		+ \sum_{1 \leq i_0,i_2 \leq m}  f( \langle p , q_{i_0} \rangle_A \otimes  \langle p_{i_0} b_0 1_Bb_1, q_{i_2} \rangle_A) p_{i_2}
		- \sum_{1 \leq i_0,i_1 \leq m} f( \langle p , q_{i_0} \rangle_A \otimes  \langle p_{i_0} b_0 , q_{i_1} \rangle_A) p_{i_1} b_1 1_B \\
		=  \sum_{1 \leq i_0,i_1,i_2 \leq m} \left (
		\langle p,q_{i_0}\rangle_A f(\langle p_{i_0},b_0q_{i_1}\rangle_A \otimes  \langle p_{i_1},b_1q_{i_2}\rangle_A)p_{i_2} -
		f( \langle p , q_{i_0}\rangle_A \langle p_{i_0}, b_0 q_{i_1} \rangle_A \otimes  \langle p_{i_1} b_1, q_{i_2} \rangle_A) p_{i_2} \right.
		\\  \left.
		+  f( \langle p , q_{i_0} \rangle_A \otimes  \langle p_{i_0} b_0 , q_{i_1}\rangle_A \langle p_{i_1}b_1, q_{i_2} \rangle_A) p_{i_2}
		- f( \langle p , q_{i_0} \rangle_A \otimes  \langle p_{i_0} b_0 , q_{i_1} \rangle_A) \langle p_{i_1} b_1  q_{i_2}\rangle_A p_{i_2} \right )
	\end{align*}
	and the last expression vanishes since $f$ is a $2$-cocycle.
	For the second part, since $g$ is a $2$-cocycle, we have that
		\begin{align*}
		\sum_{i=1}^{m'} p'_i g(\langle q'_i, p \rangle_B b_0 \otimes  b_1) &=  \sum_{i=1}^{m'} \left( \langle p'_i,q'_i\rangle_A pg(b_0 \otimes b_1)+ p'_i g(\langle q'_i, p \rangle_B \otimes  b_0b_1)- p'_i g(\langle q'_i, p \rangle_B \otimes  b_0)b_1 \right) \\
		& =   pg(b_0 \otimes b_1)+ \sum_{i=1}^{m'}  p'_i g(\langle q'_i, p \rangle_B \otimes  b_0b_1)- \sum_{i=1}^{m'} p'_i g(\langle q'_i, p \rangle_B \otimes  b_0)b_1,
	\end{align*}
	then
	\begin{align*}
		pg(b_0 \otimes b_1)-\sum_{i=1}^{m'} p'_i g(\langle q'_i, p \rangle_B b_0 \otimes  b_1)+\sum_{i=1}^{m'} p'_i g(\langle q'_i, p \rangle_B \otimes b_0b_1)- \sum_{i=1}^{m'} p'_i g(\langle q'_i, p \rangle_B \otimes  b_0)b_1 =0.
	\end{align*}
	The same conclusion can be drawn for $\sum_{i=1}^{m} f(a \otimes \langle p, q_i \rangle_A) p_i$ and equation \eqref{eq:bimod1}. Also, similar arguments prove that $\sum_{1 \leq i_0, i_1 \leq m'} p'_{i_0} g( \langle q'_{i_0}, a p'_{i_1} \rangle_B  \otimes \langle q'_{i_1} , p \rangle_B) $  satisfies
	\begin{align*}
		a_0 f_P(a_1  \otimes p) - f_P(a_0 a_1  \otimes p) + f_P(a_0  \otimes a_1p) - \psi^2(g)(a_0  \otimes a_1) p =0.
	\end{align*}
	Nevertheless, $f$ does not coincide with $\psi^2(g)=\psi^2(\phi{^2}(f))$.
	But, by \eqref{eq: h homotopy}, $f - \psi{^2}(\phi{^2}(f)) = d^2 h^2(f)$ and  the map $h^2(f)$ satisfies
	\begin{align*}
		a_0 h^2(f) (a_1)p- h^2(f) (a_0a_1)p+h^2(f)(a_0)a_1p- d^2h^2(f)(a_0 \otimes a_1)p=0.
	\end{align*}
	Finally, equation \eqref{eq:bimod3} follows since $f$ and $g$ are cocycles. Indeed, 
	\begin{multline*}
		\sum_{1 \leq i_0,i_1 \leq m} af( \langle p, q_{i_0} \rangle_A  \otimes \langle p_{i_0} b, q_{i_1} \rangle_A) p_{i_1} - \sum_{1 \leq i_0,i_1 \leq m} f( \langle ap, q_{i_0} \rangle_A  \otimes \langle p_{i_0} b, q_{i_1} \rangle_A) p_{i_1} \\
		+ \sum_{i_0=1}^m f(a  \otimes \langle p 1_B b, q_{i_0} \rangle_A) p_{i_0}
		- \sum_{i_0=1}^m f( a  \otimes\langle p, q_{i_0} \rangle_A)  p_{i_0} b 1_B \\
		+  \sum_{i_0=1}^{m'} 1_Aap'_{i_0}  g(\langle q'_{i_0},p\rangle_B  \otimes b) -  \sum_{i_0=1}^{m'} p'_{i_0}g(\langle q'_{i_0}, a 1_Ap\rangle_B  \otimes b)\\
		+ \sum_{1 \leq i_0,i_1 \leq m'}   p'_{i_0} g(\langle q'_{i_0}, ap'_{i_1}\rangle_B  \otimes \langle q'_{i_1}, p\rangle_B b) -  \sum_{1 \leq i_0,i_1 \leq m'}  p'_{i_0} g(\langle q'_{i_0}, ap'_{i_1}\rangle_B  \otimes \langle q'_{i_1}, p\rangle_B)b \\
		+  h^2(f)(a)(pb)- h^2(f)(a)p b= 0.
	\end{multline*}
\end{proof}

Now we shall describe the uple associated to  the bimodule $\hat P \otimes_{B_g} \hat Q$, and we will use this description to see that $\hat P \otimes_{B_g} \hat Q$ and $A_f$ are isomorphic as $A_f$-bimodules.  An analogous proof will lead to the same result for  $\hat Q \otimes_{A_f} \hat P$ and $B_g$.

First observe that the following equalities hold in $\hat P \otimes_{B_g} \hat Q$,  for all $x\in P$, $y\in Q$, $b\in B$:
\begin{align} \label{eq:cond1}
	(0,x) \otimes (0,y) &= (x,0)(0,1) \otimes (0,1)(y,0)=0,  \\
	(0,x) \otimes (y,0) &= (x,0) (0,1) \otimes (y,0) = (x,0) \otimes (0,y),  \label{eq:cond2} \\
	(xb,   g_P(x  \otimes b)) \otimes (y,0) & = (x,0)(b,0) \otimes (y,0) = (x,0) \otimes (by,  g_P(b  \otimes y)).\label{eq:cond4}
\end{align}
From \eqref{eq:cond1} and \eqref{eq:cond2} we can deduce that elements in $\hat P \otimes_{B_g} \hat Q$ can be expressed as
\begin{align*}
	\sum_i (x_i,0) \otimes (y_i,0) + \sum_j (x_j,0) \otimes (0,y_j).
\end{align*}
If $ (Z_0, Z_1, T_Z:Z_0 \to Z_1, f_Z: A \otimes_\K Z_0 \to Z_1, g_Z: Z_0 \otimes_\K A \to Z_1)$
is the uple associated to the bimodule $\hat P \otimes_{B_g} \hat Q$ as described in Remark \ref{bimodulos}, we know that $Z_1$ is the kernel of the $\K$-linear map
$$T: \hat P \otimes_{B_g} \hat Q \to \hat P \otimes_{B_g} \hat Q$$
given by $T(z) = (0,1) z$, for any $z\in \hat P \otimes_{B_g} \hat Q$. It is clear that
\begin{align*}
	Z_1= \Ker T = \{ \sum_j (x_j,0) \otimes (0,y_j) \} \simeq P \otimes_B Q
\end{align*}
where the last isomorphism follows because, for all $x\in P$, $y\in Q$, $b\in B$, we have
$$(xb,0) \otimes (0,y) = (0,xb) \otimes (y,0) = (x,0) (0,b) \otimes (y,0) = (x,0) \otimes (0,by).$$
Remark \ref{key} allows us to show that
\begin{align*}
	Z_0 = \{ \sum_{j} (x_j, \sum_{i=1}^{m'} g_P( p'_i  \otimes \langle q'_i, x_j \rangle_B) )\otimes (y_j, 0) \}
\end{align*}
is the complement of $Z_1$ in the $\K$-vector space $\hat P \otimes_{B_g} \hat Q$.
Also
$$T(\sum_{j} (x_j, \sum_{i=1}^{m'} g_P( p'_i  \otimes  \langle q'_i, x_j \rangle_B) )\otimes (y_j, 0))= \sum_j (0,x_j) \otimes (y_j,0)$$
and hence the map $T_Z :Z_0 \to Z_1$ induced by $T$ is an isomorphism, $Z_0 \cap Z_1= \{0\}$  and $\hat P \otimes_{B_g} \hat Q = Z_0 \oplus Z_1$ by a dimension argument.

The left $A$-module structure of $Z_0$ is given  by \eqref{eq:action A in M0}, that is, if $a \in A$ and
$$z_0=\sum_{j} (x_j, \sum_{i=1}^{m'} g_P( p'_i  \otimes \langle q'_i, x_j \rangle_B))\otimes (y_j, 0) \in Z_0,$$
then	
\begin{align*}
	a*z_0 &=T_Z^{-1}(a\cdot T_Z(z_0))	= T_Z^{-1} (\sum_{j} (0,ax_j)\otimes (y_j,0)) \\
	& = \sum_j (ax_j,  \sum_{i=1}^{m'} g_P( p'_i  \otimes \langle q'_i, ax_j \rangle_B ))\otimes (y_j,0),
\end{align*}
and the right $A$-module structure is induced by the $A_f$-module structure of $\hat Q$.
Hence $T_Z$ is an isomorphism of $A$-bimodules and $Z_0 \simeq Z_1 \simeq P\otimes_BQ$ as $A$-bimodules.

From \eqref{eq:f_M} we deduce that $f_Z: A \otimes_\K Z_0 \to Z_1$ is,  for all $a\in A$, $z_0\in Z_0$,  given by
\begin{align*}
	f_Z(a\otimes z_0) &=(a,0)z_0-a*z_0  = \sum_{j} (ax_j,  \sum_{i=1}^{m'} ag_P(p'_i  \otimes \langle q'_i, x_j \rangle_B)+ f_P(a  \otimes x_j))\otimes (y_j,0)\\
	& - \sum_{j} (ax_j,  \sum_{i=1}^{m'} g_P(p'_i  \otimes \langle q'_i, ax_j \rangle_B))\otimes (y_j,0)\\
	&=   \sum_{j} \left( 0, \sum_{i=1}^{m'}  ag_P(p'_i  \otimes \langle q'_i, x_j \rangle_B)- \sum_{i=1}^{m'} g_P(p'_i  \otimes \langle q'_i, ax_j \rangle_B )  +f_P(a  \otimes x_j) \right) \otimes (y_j,0).
\end{align*}
Analogously, we can construct $g_Z: Z_0 \otimes_\K B \to Z_1$.

Having described the bimodule $\hat P \otimes_{B_g} \hat Q$  in an appropriate way, we are now ready to construct the isomorphism needed for the Morita equivalence.

\begin{proposition}\label{lem:iso_Af_Bg}
	The $A_f$-bimodule $\hat P \otimes_{B_g} \hat Q$ is isomorphic to $A_f$ and the $B_g$-bimodule $\hat Q \otimes_{A_f} \hat P$ is isomorphic to $B_g$.
\end{proposition}
\begin{proof}
	We start by constructing a morphism
	\begin{align*}
		w: (Z_0, Z_1, T_Z, f_Z, g_Z) \to (A,A,\Id, f, g).
	\end{align*}
	That is, $w=(w_0, w_1, w_2)$ where $w_0 \in \Hom_A(Z_0, A), w_1 \in \Hom_\K (Z_0, A)$ and $w_2 \in \Hom_A (Z_1, A)$ are such that  the diagram
	\begin{align*}
		\xymatrix{
			Z_0 \ar[d]_{T_Z} \ar[r]^{w_0} & A  \ar[d]^{\operatorname{\Id}}\\
			Z_1  \ar[r]^{w_2}  & A }
	\end{align*}
	commutes and, for $z_0 \in Z_0$,
	\begin{align}\label{eq:morfismo}
		w_1(a*z_0) = a w_1 (z_0) - w_2 (f_Z(a \otimes z_0)) + f(a \otimes w_0(z_0)).
	\end{align}
	For
	$$z_0=\sum_{j} (x_j,  \sum_{i=1}^{m'} g_P( p'_i\otimes \langle q'_i, x_j \rangle_B)) \otimes (y_j, 0),$$
	the desired morphisms are defined by
	\begin{align}
			w_0(z_0) &= \sum_{j} \langle x_j,y_j\rangle_A=w_2(\sum_j (x_j,0) \otimes (0,y_j)),
		\\
		 w_1(z_0) &= \sum_{j} \left( \sum_{i=1}^{m'} \langle g_P(p'_i\otimes \langle q'_i, x_j\rangle_B), y_j\rangle_A  \right.
		 - \sum_{k=1}^{m} \langle f_P( \langle x_j, q_k\rangle_A \otimes p_k), y_j\rangle_A
		\\ \notag
		& \qquad  \left. + \sum_{k=1}^{m} f( \langle x_j, q_k \rangle_A \otimes \langle p_k, y_j\rangle_A) \right)
	\end{align}
	where the last formula is suggested by Remark \ref{key}.
	The commutativity of the diagram is immediate, and equality \eqref{eq:morfismo} holds since
	using  \eqref{eq:bimod1} and that $f$ is a cocycle, for all $a\in A$, $z_0\in Z_0$, we have
	\begin{align*}
		w_1(a*z_0) - a w_1 (z_0) =& \sum_{j} \sum_{i=1}^{m'} \langle g_P(p'_i\otimes \langle q'_i, ax_j\rangle_B), y_j\rangle_A
		- a \sum_{j} \sum_{i=1}^{m'}\langle g_P(p'_i\otimes \langle q'_i, x_j\rangle_B), y_j\rangle_A \\
		& + \sum_{j} \sum_{k=1}^{m}\langle f( a \otimes \langle x_j, q_k\rangle_A) p_k, y_j\rangle_A
		- \sum_{j} \langle f_P( a \otimes x_j), y_j\rangle_A \\
		& + \sum_{j} f(a,  \langle x_j, y_j\rangle_A)
		- \sum_{j} \sum_{k=1}^{m}f( a \otimes \langle x_j, q_k \rangle_A)  \langle p_k, y_j\rangle_A,
	\end{align*}
	where the third term cancels with the last one.
	
	On the other hand,  for all $a\in A$, $z_0\in Z_0$, the equality
	\begin{align*}
		- w_2 (f_Z(a \otimes z_0)) + f(a \otimes w_0(z_0))  = &- a \sum_{j} \sum_{i=1}^{m'} \langle g_P(p'_i\otimes \langle q'_i, x_j\rangle_B), y_j\rangle_A  \\
		&  + \sum_{j} \sum_{i=1}^{m'} \langle g_P(p'_i\otimes \langle q'_i, ax_j\rangle_B), y_j\rangle_A  \\
		& - \sum_{j} \langle f_P( a \otimes x_j), y_j\rangle_A + \sum_{j} f(a,  \langle x_j, y_j\rangle_A)
	\end{align*}
	holds and completes the first part of the proof.
	
	It only remains to prove that $w=(w_0, w_1, w_2)$ is an isomorphism.  A direct computation shows that
	\begin{align*}
		(w_0^{-1}, - w_2^{-1} w_1 w_0^{-1}, w_2^{-1}):  (A,A,\Id, f, g) \to (Z_0, Z_1, T_Z, f_Z, g_Z)
	\end{align*}
	is a morphism, and it is the inverse of $w$. 
\end{proof}

\begin{theorem}
	Let $A$ and $B$ be Morita equivalent $\K$-algebras. Let $[f]\in\HH^2(A)$ and $[g]\in\HH^2(B)$ connected via the isomorphism $\HH^2(A)\simeq\HH^2(B)$. Then, the infinitesimal deformations $A_f$ and $B_g$ are Morita equivalent.
\end{theorem}
\begin{proof}
	By Lemma \ref{lem:hatP and hatQ} we have bimodules $\hat P$ and $\hat Q$ that give us the Morita equivalence by Proposition \ref{lem:iso_Af_Bg}.
\end{proof}

\begin{remark} From Remark \ref{stable} one can infer that 
the results obtained in this section for Morita equivalent algebras could probably be generalized to the stable equivalence of Morita type context. For that, it would be necessary to take special care in proving that the bimodules $\hat P$ and $\hat Q$ are projective as well as showing the projectivity of the direct sumands $\hat X$ and $\hat Y$ appearing in $\hat P \otimes_{B_g} \hat Q\simeq A_f\oplus \hat X$ and  $\hat Q \otimes_{A_f} \hat P\simeq B_g\oplus \hat Y$.
\end{remark}

\section{Quiver associated to $A_f$}\label{sec:quiver Q_f}

Let $A$ be a finite dimensional algebra over an algebraically closed field $\K$ and let $f : A \otimes_k A \to A$ be a Hochschild 2-cocycle.  The aim of this section is to  describe the presentation by  quiver and relations for $A_f$.
Since deformations behave well with Morita equivalence, and $A$ is Morita equivalent to a quotient of a path algebra, we may assume that $A=\K Q/I$. 

\subsection{Presentation by quiver and relations of $A_f$} Set $A=\K Q/I$.
Since the set of equivalence classes of paths in $Q$ generates $A$, we can fix a set $\P$ of paths in $Q$ such that $\{ \og \colon \gamma \in \P \}$ is a basis of $A$. Let $\Gamma= \{\rho_1, \cdots, \rho_m\}$ be a set of relations generating $I$. 

As we mentioned in Subsection \ref{subsect: quiver}, we say that an arrow $\alpha$ divides a path $w$, and denote $\alpha / w$, if $w= L_\alpha(w)\alpha R_\alpha(w)$, where $L_\alpha(w)$ and $R_\alpha(w)$ are paths in $Q$. 

Since $\K$ is a field, $\HH^n(A) \simeq \Ext^n_{A-A}(A,A)$, and hence one can use any projective resolution of the $A$-bimodule $A$ to compute this cohomology. As we said in Subsection \ref{subsect:HochCohom}, we may assume that $f \in \Hom_{E-E}(\rad A \otimes_E \rad A, A)$, with $E=\K Q_0$. This holds since $(A \otimes_E \rad A^* \otimes_E A, d_*)$ is a projective resolution of the $A$-bimodule $A$, see  \cite[Lemma 2.1]{Cibils}. 
However, we can construct a more convenient projective resolution by following the method developed in \cite{RR2} when dealing with monomial algebras.  Even though we are now working in a more general setting, the first degrees behave in the same way, and since we are interested in the second cohomology group, it is enough to describe this resolution up to degree $2$.

Define the sequence
$$(S_*, \delta_*):  \qquad  \cdots  \to A \otimes_E I \otimes_E A \stackrel{\delta_{2}}{\longrightarrow}A \otimes_E \K Q_1 \otimes_E A \stackrel{\delta_{1}}{\longrightarrow} A \otimes_E \K Q_0 \otimes_E A \stackrel{\delta_{0}}{\longrightarrow} A \to 0$$ 
where, for any $\alpha \in Q_1$ and  $\sum_i \lambda_i w_i \in I$, with all $w_i$ paths in $Q$,
\begin{align*}
	\delta_0(1 \otimes e_i \otimes 1) & = e_i,\\
	\delta_1(1 \otimes  \alpha\otimes  1) & = \oa \otimes e_{t(\alpha)} \otimes 1 - 1    \otimes e_{s(\alpha)} \otimes \oa,  \\
	\delta_2 (1 \otimes  \sum_i \lambda_i w_i \otimes 1) &= \sum_i \lambda_i \sum_{\alpha \in Q_1: \alpha/w_i} \oL_\alpha(w_i) \otimes \alpha \otimes \oR_\alpha(w_i).
\end{align*}

\begin{lemma}
	The sequence $(S_*, \delta_*)$ is the starting point of a projective resolution of the $A$-bimodule $A$. Moreover, the $E$-$A$-bilinear maps 
	\begin{align*}
		c_0: & A  \to A \otimes_E \K Q_0 \otimes_E A, \\
		c_1:  & A \otimes_E \K Q_0 \otimes_E A   \to A \otimes_E \K Q_1 \otimes_E A, \\
		c_2: & A \otimes_E \K Q_1 \otimes_E A   \to  A \otimes_E I \otimes_E A,
	\end{align*}
	defined by
	\begin{align*}
		c_0(1) & = \sum_{i\in Q_0} 1 \otimes e_i \otimes 1,\\
		c_1 (\og \otimes e_{t(\gamma)} \otimes 1) & = \sum_{ \alpha \in Q_1: \alpha / \gamma } \oL_\alpha(\gamma) \otimes \alpha \otimes \oR_\alpha(\gamma), \qquad & \mbox{for any $\gamma \in \P$},\\
		c_2 (\og \otimes \alpha \otimes 1)& =  1 \otimes (\gamma \alpha- \sum_i \lambda_i { \gamma _i}) \otimes 1, &\mbox{for any $\gamma \in \P$,  $\gamma \alpha - \sum_i \lambda_i { \gamma_i} \in I$,$ {\gamma_i \in \P}$},
	\end{align*}
	is a contracting homotopy in the first degrees.
\end{lemma}

\begin{proof}
Since $E$ is semisimple, for any $E$-bimodule $X$ we have that $A\otimes_E X \otimes_E A$ is a projective $A$-bimodule.
	A direct computation shows that $\delta_0 \delta_1=0$ and $\delta_1 \delta_2=0$. Concerning the homotopy, for all $\gamma\in \P$, we have
	\begin{align*}
		&\delta_0 c_0 (\og)  = \delta_0 (1 \otimes e_{s(\gamma)} \otimes \og) = \og, \\
		&c_0 \delta_0  (\og \otimes e_{t(\gamma)} \otimes 1)   = c_0 (\og)  = 1  \otimes e_{s(\gamma)} \otimes \og,
	\end{align*}
	\begin{align*}
		\delta_1 c_1 (\og \otimes e_{t(\gamma)} \otimes 1) & = \delta_1 (\sum_{ \alpha \in Q_1: \alpha / \gamma } \oL_\alpha(\gamma) \otimes \alpha \otimes \oR_\alpha(\gamma)) \\ \notag
		& = \sum_{ \alpha \in Q_1: \alpha / \gamma } (\oL_\alpha(\gamma)  \oa \otimes e_{t(\alpha)} \otimes \oR_\alpha(\gamma) - \oL_\alpha(\gamma) \otimes e_{s(\alpha)} \otimes  \oa \oR_\alpha(\gamma)) \\ \notag
		& = \og  \otimes e_{t(\gamma)}  \otimes 1 - 1  \otimes e_{s(\gamma)} \otimes \og,
	\end{align*}
	that is,  $\delta_0 c_0= \Id$ and $c_0 \delta_0 + \delta_1 c_1= \Id$. Finally, let $\gamma\in\P$ and $\alpha\in Q_1$; if $\og \oa = \sum_i \lambda_i \og_i$  with $\gamma_i \in \P$,  then
	\begin{align*}
		c_1 \delta_1  (\og \otimes  \alpha\otimes  1)& = c_1 ( \og \oa \otimes e_{t(\alpha)} \otimes 1 - \og \otimes e_{s(\alpha)}\otimes \oa) \\ \notag
		& = \sum_i \lambda_i c_1 ( \og_i  \otimes e_{t(\gamma_i)} \otimes 1) - c_1(\og \otimes e_{s(\alpha)}\otimes \oa) \\ \notag
		& = \sum_i \lambda_i \sum_{ \beta \in Q_1: \beta / \gamma_i } \oL_\beta(\gamma_i) \otimes \beta \otimes \oR_\beta(\gamma_i) - \sum_{ \beta \in Q_1: \beta / \gamma } \oL_\beta(\gamma) \otimes \beta \otimes \oR_\beta(\gamma) \oa, \\
		\delta_2 c_2   (\og \otimes  \alpha\otimes  1)& = \delta_2 (1 \otimes (\gamma \alpha- \sum_i \lambda_i \gamma_i) \otimes 1) \\ \notag
		& = \sum_{ \beta \in Q_1: \beta / \gamma \alpha } \oL_\beta(\gamma \alpha) \otimes \beta \otimes \oR_\beta(\gamma \alpha)  -  \sum_i \lambda_i \sum_{ \beta \in Q_1: \beta / \gamma_i } \oL_\beta(\gamma_i) \otimes \beta \otimes \oR_\beta(\gamma_i).
	\end{align*} 
	Hence $(c_1 \delta_1 + \delta_2 c_2) (\og \otimes  \alpha\otimes  1) = \og \otimes  \alpha\otimes  1$ and the proof is done.
\end{proof}

\medskip

Since $(S_*, \delta_*)$ and $(A \otimes_E \rad A^* \otimes_E A, d_*)$ are projective resolutions of the $A$-bimodule $A$, a well known result ensures the existence of comparison morphisms
$$F_*: S_* \to A \otimes_E \rad A^* \otimes_E A \quad \mbox{ and } \quad  G_*: A \otimes_E \rad A^* \otimes_E A \to S_*.$$
Again from \cite{RR2} we can deduce the formula for  the comparison morphism $F$ in the first degrees.  For $\alpha \in Q_1$, $ \sum_j \lambda_j \ \alpha^j_1 \cdots \alpha^j_{s_j} \in I$, 
\begin{align*}
	F_0(1 \otimes e \otimes 1) & = e \otimes 1, \\ 
	F_1(1 \otimes \alpha \otimes 1) & = 1 \otimes \oa  \otimes 1, \\
	F_2 (1 \otimes  \sum_j \lambda_j \ \alpha^j_1 \cdots \alpha^j_{s_j} \otimes 1) & = \sum_j \lambda_j \sum_{i=1}^{s_j-1} 1 \otimes \oa^j_1 \cdots \oa^j_{i} \otimes \oa^j_{i+1} \otimes \oa^j_{i+2} \cdots \oa^j_{s_j} .
\end{align*}

Using the identification
$$ \Hom_{E-E}(\rad A \otimes_E \rad A, A) \simeq \Hom_{A-A}(A \otimes \rad A^{\otimes_2 } \otimes A, A)$$  given by $\tilde{f}(1\otimes w_1 \otimes w_2 \otimes 1) = f(w_1 \otimes w_2)$, 
we have that if  $f \in \Hom_{E-E}(\rad A \otimes_E \rad A, A)$  is a Hochschild 2-cocycle,  $\tilde{f} \circ F_2 \in \Hom_{E-E}(A \otimes I  \otimes A , A)$ is also a $2$-cocycle and they represent the same element in $\HH^2(A)$. 

In the following definition we will introduce a new map $\hat f : \K Q \to A$ associated to $f$ that will be needed in the main theorem of this section.

\begin{definition}\label{defi} Let $\hat f : \K Q \to A$ be the $\K$-linear map defined in paths as follows:
	\[
	\hat f (\alpha_1 \cdots \alpha_s)  =
	\begin{cases}
	0,    & \mbox{if $s=0,1$,} \\
	\sum_{i= 1}^{s-1}f(  \oa_1 \cdots \oa_{i}  \otimes \oa_{i+1} )  \ \oa_{i+2} \cdots \oa_{s},  \quad & \mbox{if $s \geq 2$.}
	\end{cases}
	\]
	In particular, if $\rho \in I$ then
	$$\hat{f} (\rho) =  (\tilde{f} \circ F_2) (1 \otimes \rho \otimes 1).$$
\end{definition}

\begin{remark}\label{rem: imagen f}
	For any Hochschild $2$-cocycle $f$ we know that $[f] = [\tilde f F_2 G_2]$ since $F_* G_*$ is homotopic to the identity.
	Hence we can replace $f$ by $\tilde f F_2 G_2$, and since $\Im \tilde f F_2 G_2  \subset  \hat f (I)$, from now on we will assume that the representative $f$ we have chosen satisfies $\Im f \subset \hat f (I)$.
\end{remark} 
Next lemma shows the relation between the map $\hat f$ and the product in $A_f$.

\begin{lemma} \label{Im f}
	Let $w=  \alpha_1 \cdots \alpha_{s}$ be a path in $Q$ with $s \geq 1$. Then  the equality 
	\begin{align*}
		( \oa_1,0) \cdots (\oa_{s},0)  &= (\ow , \hat f(w))
	\end{align*}
	holds in $A_f$. In particular, if $\rho= \sum_j \lambda_j \ \alpha_1^j \cdots \alpha_{s_j}^j \in I$ then
	\begin{align*}
		\sum_j \lambda_j  ( \oa^j_1,0) \cdots (\oa^j_{s_j},0)  &= (0, \hat f(\rho)).
	\end{align*}
\end{lemma}

\begin{proof} An inductive procedure on $s$ shows that 
	$$( \oa_1,0) \cdots (\oa_{s},0) = (\oa_1 \cdots \oa_{s} ,\sum_{i= 1}^{s-1}f(  \ \oa_1 \cdots \oa_{i}  \otimes \oa_{i+1} )  \ \oa_{i+2} \cdots \oa_{s}).$$
\end{proof}

Let $A=\K Q/I$ and let $[f] \in \HH^2(A)$.
Our next goal is to give a presentation by quiver and relations of $A_f$.

Recall that $\Gamma= \{\rho_1, \cdots, \rho_m\}$ is a set of relations generating $I$. For each $\rho \in \Gamma$ we fix an element $w_\rho$ in $\K \P$ such that $\overline w_\rho = \hat f (\rho)$. 
Let $Q_{f}$ be the quiver given by
\begin{align*}
	(Q_{f})_0 & = Q_0, \\ 
	(Q_{f})_1 & =  \{ \hat \alpha  \ \vert \ \alpha\in Q_1\} \cup    \{ \hat e_i : i \to i \ \vert \  e_i \not \in \Im f\}.  
\end{align*}
For any path  $w = \alpha_1 \cdots \alpha_s$  in $Q$, we set $\hat w : = \hat \alpha_1 \dots \hat \alpha_s$  in $Q_f$, and for any vertex $i \in Q_0$ we denote 
\[
\epsilon_i=  \begin{cases}
\hat e_i, & \mbox{if $e_i \not \in \Im f$}, \\
\sum_j  \mu_j \hat \rho_j,  \quad & \mbox{if $e_i=\hat f( \sum_j \mu_j \rho_j)$, $\rho_j \in \Gamma$}.
\end{cases} \]
Let  $I_{f} $ be the ideal of $\K Q_f$ generated by
\begin{align}\label{eq: rel 1}
	\epsilon_i^2&   \hspace{0.6cm} \mbox{ for } i  \in Q_0;  \\ \label{eq: rel 2}
	\epsilon_i  \hat \alpha  -  \hat \alpha  \epsilon_j&     \hspace{0.6cm} \mbox{ for } \alpha: i \to j   \in Q_1;\\ \label{eq: rel 3}
	\hat \rho  - \hat w_\rho \epsilon_{t(\rho)}&    \hspace{0.6cm} \mbox{ for } \rho  \in \Gamma.
\end{align}

\begin{theorem} 
	Let $A=\K Q/I$ and let $[f] \in \HH^2(A)$.  Then $A_f = \K Q_{f}/I_{f}$.
\end{theorem}

\begin{proof}
	Since $A_f = A \oplus A$ with $(a,b) (a',b') = (aa', f(a \otimes a') + ab'+ba')$, for $a,a',b,b'\in A$, we have that
	$$ (\rad A \oplus A )^m \subset \rad^m A \oplus (\rad^{m-1} A + \sum_{i+j=m-1} \rad^i A . f (\rad^j A  \otimes \rad A))$$
	which is clearly nilpotent, hence $ \rad A \oplus A \subset \rad A_f $.  Moreover, $A_f / (\rad A \oplus A) \simeq A/ \rad A$ is semisimple, hence $\rad A_f  = \rad A \oplus A$ and $ (Q_{A_f})_0  = (Q_A)_0$.
	Now we will prove that $$\rad^2 A_f = \rad^2 A \oplus (\rad A + \Im f).$$  The first inclusion is clear.
	On the other hand, by Remark \ref{rem: imagen f} we may assume $\Im f \subset \hat f (I)$. Thus,  for $r_1, r_2 \in \rad A$ we have
	\begin{align*}
		(r_1 r_2, 0) & =  (r_1, 0)(r_2, 0) - (0, f(r_1 \otimes r_2)) \\
		& = (r_1, 0)(r_2, 0) - (0, \hat f (\rho) ),
	\end{align*}	
	for some $\rho \in I$, and from Lemma \ref{Im f} we conclude that $(r_1 r_2, 0) \in \rad^2 A_f$. Similarly, if $\hat f(\rho) \in \Im f$ for some $\rho \in I$ then
	$(0, \hat f (\rho)) \in \rad^2 A_f$. Finally, if $r \in \rad A$
	$$(0,r) =(r,0) (0, e_{t(r)}) \in \rad^2 A_f.$$
	Then  $\rad A_f / \rad^2 A_f= \rad A / \rad^2 A \oplus A / (\rad A + \Im f)$ and the description of $(Q_{f})_1$ is clear. 
	
	Let $\pi: \K Q_{f} \to A_f$ be the algebra epimorphism given by $\pi(i) = (e_i,0)$, $\pi (\hat \alpha) = (\oa,0)$ and $\pi(\hat e_i) = (0, e_i)$.  It is clear that $I_{f}  \subset \Ker \pi$ since, by Lemma \ref{Im f},
	\begin{align*}
		\pi (\epsilon_i^2) &=\begin{cases}
			\pi ( \hat e_i)^2=  (0,e_i)^2 =0,  & \mbox{if $e_i \not \in \Im f$}, \\
			\pi (\sum_j  \mu_j \hat \rho_j)^2 = (0,e_i)^2 = 0, \quad & \mbox{if $e_i=\hat f( \sum_j \mu_j \rho_j)$, $\rho_j \in \Gamma$},
		\end{cases} \\ 
		\pi (\epsilon_i  \hat \alpha-  \hat \alpha \epsilon_j) & = (0, e_i) (\bar{\alpha}, 0) -  (\bar{\alpha}, 0) (0, e_j) =0.
	\end{align*}
	Also, if $\rho  \in \Gamma$,  we get 
	\begin{align*}
		\pi  (\hat \rho) & =  (0 , \hat f(\rho))  = (\hat f(\rho),0)(0, e_{t(\rho)}) = \pi (\hat w_\rho \epsilon_{t(\rho)}).	
	\end{align*}
	It remains to prove that $\Ker \pi \subset I_f$. Consider the set $\mathcal B$ of equivalence classes of elements in 
	\[ \hat \P \cup \hat \P \epsilon= \{ \hat \gamma \  \vert \ \gamma \in \P\} \cup  \{ \hat \gamma \epsilon_{t(\gamma)}\ \vert \ \gamma \in \P\}. \]
	We claim that $\mathcal B$ is a basis of  $\K Q_f/I_f$.  Indeed,  $\K Q_f/I_f$ is generated by equivalence classes of paths in $Q_f$.  But, by \eqref{eq: rel 1} and \eqref{eq: rel 2}, we can choose  representatives  of the form $\hat \theta$ or $\hat \theta \epsilon_i$ with $\theta$  a path in $Q$.  Now, our assumption on the set $\P$ implies that any path $\theta$ can be written as
	$$\theta = \sum_j \lambda_j \gamma_j + \sum_k \mu_k \rho_k$$
	with $\gamma_j \in \P$, $\rho_k \in \Gamma$.
	Moreover,  $\hat \rho_k - \hat w_{\rho_k} \epsilon_{t(\rho_k)} \in I_f$ by \eqref{eq: rel 3}. Hence, the representatives in $\K Q_f$ shall be of the form
	\begin{align}
		\hat \theta &=
		\sum_j \lambda_j \hat \gamma_j +\sum_k \mu_k \hat w_{\rho_k} \epsilon_{t(\rho_k)}, \quad \mbox{ or }\\
		\hat \theta \epsilon &=
		\sum_j \lambda_j \hat \gamma_j  \epsilon_{t(\gamma_j)},
	\end{align}
	with all $\hat \gamma_j \in \hat \P$  and where all $\hat w_{\rho_k} \epsilon_{t(\rho_k)}$ can be written, modulo $I_f$, as a linear combination of paths in $\hat \P \epsilon$. Thus $\mathcal B$ generates  $\K Q_f/I_f$.
	
	Finally, for any  $\gamma= \alpha_1 \cdots \alpha_s \in \P$ we have
	\begin{align*}
		\pi(\hat \gamma) &=\pi(\hat \alpha_1) \cdots \pi(\hat \alpha_s) = (\bar{\alpha_1},0) \cdots (\bar{\alpha_s},0) = (\bar{\gamma} , \hat f(\gamma)),\\ 
		\pi(\hat \gamma \epsilon_i) &= \pi(\hat \gamma) \pi( \epsilon_i)= (\bar{ \gamma }, \hat f(\gamma)) (0, e_i) = (0 , \bar{\gamma}), \qquad &\mbox{ if $e_i \not \in   \Im f$},\\  
		\pi(\hat \gamma \epsilon_k) &= \pi(\hat \gamma) \pi(  \sum_j \mu_j \hat \rho_j)= (\bar{\gamma}, \hat f(\gamma)) (0, e_k)  =  (0 ,\bar{ \gamma}),  \qquad & \mbox{if $e_k=\hat f( \sum_j \mu_j \rho_j)$}. 
	\end{align*}
	Let $\pi^* : \K Q_f/I_f \to A_f$. Then 
	$\pi^* (\mathcal B)$ is a basis of  $A_f$, and therefore the set  $\mathcal B$ is linearly independent and we are done. 
\end{proof}

\subsection{Examples}
 In the examples below we use results from \cite{RR1} and  \cite{BGMS} for describing non zero elements in the second cohomology group of certain monomial and  non monomial algebras respectively.

\begin{example} Let $A= \K [t]/(t^2)$ be the algebra of dual numbers. In this case, $A \simeq \K Q/I$, where $Q$   is the quiver: 
	\[    \xymatrix{  
		1 \ar@(ur,dr)^{\alpha}   } \] 
	and $I= <\alpha^2  >$.   Let  $[f] \in \HH^2(A)$ be given by
	\begin{align*} f(a \otimes b) = \left\{
		\begin{array}{ll}
			e_1,   \qquad & \mbox{if $a= b =  \alpha$},\\
			0, & \mbox{otherwise}.
		\end{array} \right. \end{align*} 
		Thus $A_{f} \simeq \K Q_{f}/ I_{f}$ where $Q_{f}=Q$ is the same as the quiver of $A$ and $I_{f} = <\alpha^4 >$.
\end{example}

\begin{example}\label{ejemplo1}
	Let  $A = \K Q/I $ where $Q$ is the quiver:  \[    \xymatrix{
		1 \ar@<6pt>[rr]^{ \alpha_1} & &     2    \ar[ll]^{\alpha_2}  } \] and  $I = <\alpha_1\alpha_2  >$.
The dimension of the second Hochschild cohomology group of $A$ is one  and a generator $[f]$ is given by
	\begin{align*}f(a \otimes b) = \left\{
		\begin{array}{ll}
			a_1a_2,      \qquad & \mbox{if  $a b= a_1\alpha_1 \alpha_2a_2$},\\
			0, & \mbox{otherwise},
		\end{array} \right. \end{align*} where $a_1$ and  $a_2$  are paths in $A$.
	Consider  the infinitesimal deformation  $A_{f}$. From  our previous theorem we have   that $A_{ f} \simeq \K Q_{ f}/I_{ f}$ where $Q_{ f}$ is the quiver
	\[    \xymatrix{
		1 \ar@<6pt>[rr]^{ \hat \alpha_1} & &     2 \ar@(ur,dr)^{\hat e_2}    \ar[ll]^{ \hat \alpha_2} } \]
	and $I_{ f} = <\hat \alpha_1\hat \alpha_2 \hat \alpha_1 \hat \alpha_2, \ \hat e_2^2, \ \hat \alpha_1 \hat \alpha_2 \hat \alpha_1 - \hat \alpha_1\hat e_2, \  \hat \alpha_2\hat \alpha_1\hat \alpha_2 -  \hat e_2\hat \alpha_2   >$.
	
\end{example}

\begin{example}  Let  $A = \K Q/I $ where $Q$ is the quiver   \[    \xymatrix{
		&  2 \ar[rd]^{ \alpha_2 } &   \\  1 \ar[rr]_{\alpha_3}  \ar[ru]^{\alpha_1 } & & 3 } \]   and $I = < \alpha_1\alpha_2>$.
	Then, $\rm dim_\K \HH^2(A)= 1$ and   $[f]$ generates $\HH^2(A)$ where 
		$$f(a \otimes b) = \left\{
	\begin{array}{ll}
	\alpha_3,    \qquad & \mbox{if $a = \alpha_1, \  b= \alpha_2$}, \\
	0, & \mbox{otherwise}.
	\end{array} \right. $$ 
	In this case,  the presentation $(Q_{f}, I_{f})$ of the 
	infinitesimal deformation  $A_{f}$  is given by
	
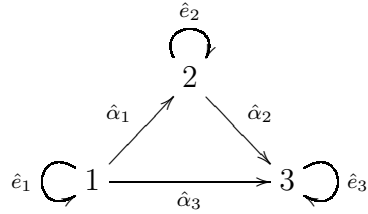
\[    \xymatrix{
		&   2 \ar@(ul,ur)^{\hat e_2} \ar[rd]^{ \hat \alpha_2} &   \\  1 \ar@(ul,dl)_{\hat e_1}  \ar[rr]_{\hat \alpha_3}  \ar[ru]^{\hat \alpha_1} & & 3 \ar@(ur,dr)^{\hat e_3} } \]
	and $I_{f} = <\hat e_1^2, \ \hat e_2^2, \  \hat e_3^2, \ \hat e_1 \hat \alpha_1 - \hat \alpha_1 \hat e_2, \  \hat e_2 \hat \alpha_2 - \hat \alpha_2\hat e_3,\  \hat e_1 \hat \alpha_3 - \hat \alpha_3 \hat e_3,  \hat \alpha_1 \hat \alpha_2 -   \hat \alpha_3\hat e_3 \    >$.
	
\end{example}

The following example deals with a non monomial algebra. For the computation of the second cohomology group we refer to \cite{BGMS}.
\begin{example} Let  $A = \K Q/I$, where $Q$  is the quiver 
	
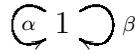
\[    \xymatrix{  
		\ar@(ul,dl)^{\alpha}   1 \ar@(ur,dr)^{\beta}   } \] 
	and $I= <\alpha^2, \beta^2, \alpha\beta+q\beta\alpha  >$, with $q\in\K$.  Let $f:A\otimes_\K A\to A$ be given by 
	\begin{align*} 
		f(\alpha \otimes \beta+ q\beta\otimes \alpha) = \beta\alpha
	\end{align*} 
	and zero otherwise. By \emph{loc. cit.}, $[f]\in\HH^2(A)$.
	Thus,  the presentation $(Q_{f}, I_{f})$ of the 
	infinitesimal deformation  $A_{f}$  is given by
	
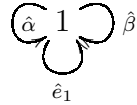
\[  \xymatrix{  
		\ar@(ul,dl)^{\hat\alpha}   1 \ar@(ur,dr)^{\hat\beta} \ar@(dl,dr)_{\hat{e}_1}  } \]
	and $I_{f} = <\hat e_1^2, \,  \hat\alpha^2, \, \hat\beta^2,\, \hat e_1\hat\alpha - \hat\alpha \hat e_1, \, \hat e_1\hat\beta - \hat\beta \hat e_1  , \, \hat\alpha\hat\beta+q\hat\beta\hat\alpha - \hat\beta\hat\alpha\hat e_1>$.
\end{example}


\begin{thebibliography}{XXXX}

\bibitem [ASS]{ASS} I. Assem, D. Simson\ and\ A. Skowro\'{n}ski, {\it Elements of the representation theory of associative algebras. Vol. 1}, London Mathematical Society Student Texts, 65, Cambridge University Press, Cambridge, 2006. MR2197389.

\bibitem [ARS]{ARS}  M. Auslander, I. Reiten\ and\ S. O. Smal\o, {\it Representation theory of Artin algebras}, corrected reprint of the 1995 original, Cambridge Studies in Advanced Mathematics, 36, Cambridge University Press, Cambridge, 1997. MR1476671.

\bibitem[B]{B} S.Bouc, \emph{Bimodules, trace g\'en\'eralis\'ee, et transferts en homologie de Hochschild}. Preprint available at \url{http://www.lamfa.u-picardie.fr/bouc/transfer.pdf}.

\bibitem[BGMS]{BGMS} R.-O. Buchweitz, E. L. Green, D. Madsen\ and\ O. Solberg, Finite Hochschild cohomology without finite global dimension, Math. Res. Lett. {\bf 12} (2005), no.~5-6, 805--816. MR2189240.

\bibitem[C]{Cibils} C. Cibils, Rigidity of truncated quiver algebras, Adv. Math. {\bf 79} (1990), no.~1, 18--42. MR1031825.
	
\bibitem[G]{G1} M. Gerstenhaber, On the deformation of rings and algebras, Ann. of Math. (2) {\bf 79} (1964), 59--103. MR0171807.

\bibitem[K]{Keller} B. Keller, Invariance and localization for cyclic homology of DG algebras, J. Pure Appl. Algebra {\bf 123} (1998), no.~1-3, 223--273. MR1492902.

\bibitem[KLZ]{KLZ}  S. Koenig, Y. Liu\ and\ G. Zhou, Transfer maps in Hochschild (co)homology and applications to stable and derived invariants and to the Auslander--Reiten conjecture, Trans. Amer. Math. Soc. {\bf 364} (2012), no.~1, 195--232. MR2833582.

\bibitem[Li]{Li} M. Linckelmann, Transfer in Hochschild cohomology of blocks of finite groups, Algebr. Represent. Theory {\bf 2} (1999), no.~2, 107--135. MR1702272.


\bibitem[Lo]{loday} J.-L. Loday, {\it Cyclic homology}, second edition, Grundlehren der Mathematischen Wissenschaften, 301, Springer-Verlag, Berlin, 1998. MR1600246.


\bibitem[RR1]{RR1} M. J. Redondo\ and\ L. Rom\'{a}n, Gerstenhaber algebra structure on the Hochschild cohomology of quadratic string algebras, Algebr. Represent. Theory {\bf 21} (2018), no.~1, 61--86. MR3748354.

\bibitem[RR2]{RR2} M. J. Redondo\ and\ L. Rom\'{a}n, Comparison morphisms between two projective resolutions of monomial algebras, Rev. Un. Mat. Argentina {\bf 59} (2018), no.~1, 1--31. MR3825761.


\bibitem[W]{W} C. A. Weibel, {\it An introduction to homological algebra}, Cambridge Studies in Advanced Mathematics, 38, Cambridge University Press, Cambridge, 1994. MR1269324.

\bibitem[Z]{Z}  A. Zimmermann, {\it Representation theory}, Algebra and Applications, 19, Springer, Cham, 2014. MR3289041.




	
\end{thebibliography}
\end{document}